\documentclass{amsart}     
\usepackage{graphicx}
\usepackage{amssymb}
\usepackage{latexsym}
\usepackage{amsfonts}
\usepackage{amsmath}
\usepackage{amsthm}

\DeclareMathOperator{\Aut}{Aut}
\DeclareMathOperator{\Sym}{Sym}
\DeclareMathOperator{\Rep}{Rep}

\DeclareMathOperator{\PSL}{PSL}

\DeclareMathOperator{\PGL}{PGL}
\DeclareMathOperator{\PSU}{PSU}

\DeclareMathOperator{\PGaL}{P\Gamma L}
\DeclareMathOperator{\AGaL}{A\Gamma L}
\DeclareMathOperator{\AGL}{AGL}

\DeclareMathOperator{\soc}{soc}

\renewcommand{\b}{\mathbf}

\renewcommand{\leq}{\leqslant}
\renewcommand{\geq}{\geqslant}

\newcommand{\F}{\mathbb F}
\newcommand{\D}{\mathcal D}
\renewcommand{\O}{\mathcal O}
\renewcommand{\H}{\mathcal H}
\renewcommand{\P}{\mathcal P}
\newcommand{\E}{\mathcal E}

\DeclareMathOperator{\supp}{supp}
\DeclareMathOperator{\Diff}{Diff}
\DeclareMathOperator{\wt}{wt}

\theoremstyle{plain}
\newtheorem{lemma}{Lemma}
\newtheorem{corollary}[lemma]{Corollary}
\newtheorem{theorem}[lemma]{Theorem}
\newtheorem{proposition}[lemma]{Proposition}

\theoremstyle{definition}
\newtheorem{definition}[lemma]{Definition}

\newtheorem{remark}[lemma]{Remark}

\numberwithin{equation}{section}
\numberwithin{lemma}{section}

\begin{document}

\title{Entry-Faithful $2$-Neighbour Transitive Codes}
 \author{Neil I. Gillespie, Michael Giudici, Daniel R. Hawtin and \\Cheryl E. Praeger}
 
 \address{[Gillespie] Heilbronn Institute for Mathematical Research, School of Mathematics, Howard House, University of Bristol, BS8 1SN, United Kingdom.\newline
 \indent[Giudici, Hawtin, Praeger] Centre for the Mathematics of Symmetry and Computation, University of Western Australia, 35 Stirling Highway, Crawley, WA 6009, Australia.}

\email{neil.gillespie@bristol.ac.uk \\michael.giudici@uwa.edu.au\\ \newline daniel.hawtin@research.uwa.edu.au \\cheryl.praeger@uwa.edu.au}

\date{\today}

\thanks{
{\it Completely transitive codes, Regular codes, $2$-Neighbour transitive codes, Automorphisms groups, Hamming graph.}\\
The research for this paper was supported by a grant associated with Australian Research Council Federation Fellowship FF0776186. The third author is supported by an Australian Postgraduate Award and UWA Top-up Scholarship.}

\begin{abstract}
 We consider a code to be a subset of the vertex set of a Hamming graph. The set of $s$-neighbours of a code is the set of vertices, not in the code, at distance $s$ from some codeword, but not distance less than $s$ from any codeword. A $2$-neighbour transitive code is a code which admits a group $X$ of automorphisms which is transitive on the $s$-neighbours, for $s=1,2$, and transitive on the code itself. We give a classification of $2$-neighbour transitive codes, with minimum distance $\delta\geq 5$, for which $X$ acts faithfully on the set of entries of the Hamming graph.
\end{abstract}

\maketitle

\section{Introduction}\label{intro}

In the context of this paper, a \emph{code} is a subset of vertices of a Hamming graph $H(m,q)$, and $X$ is a subgroup of the full automorphism group of the Hamming graph. The set of $i$-neighbours, that is, the set of vertices at distance $i$ from the code, is denoted $C_i$. An $(X,s)$-neighbour transitive code $C$ is a code, such that $C, C_1, \ldots, C_s$ are all $X$-orbits. In this paper we begin work on classifying the family of $(X,2)$-neighbour transitive codes, considering here the case where $X$ acts faithfully on the entries of the Hamming graph (see Section~\ref{prelim}). We refer to such codes as \emph{entry-faithful} $(X,2)$-neighbour transitive codes.

Two extreme cases of $(X,s)$-neighbour transitive codes have been studied previously, when $s=1$ and when $s$ is equal to the covering radius. The covering radius is the largest value $s$ can take and gives the class of \emph{$X$-completely transitive codes}, (see \cite{Giudici1999647}). These are a generalisation of \emph{coset-completely transitive codes}, which were introduced by Sol{\'e} in 1987 \cite{sole1990completely}, and are a linear sub-family of completely transitive codes. In \cite{borges2001nonexistence}, Borges et al. classified binary coset-completely transitive codes with minimum distance at least $9$, showing that the \emph{binary repetition code} (see Definition \ref{repcode}) is the unique code in this family. Other work has been done by Borges et al.~in \cite{borges2013new,borges2000nonexistence,borgesrho1,borges2012new}. There exist completely transitive codes that are not coset-completely transitive.  For example, the first and fourth authors proved that certain Hadamard codes \cite{Gillespie20131394} and the Nordstrom-Robinson code \cite{gillespie2012nord} are completely transitive, but as they are non-linear, cannot be coset-completely transitive; also the repetition code of length $3$ over a finite field $\F_q$ for $q\geq 9$ is an example of a linear completely transitive code that is not coset-completely transitive.

The family of $(X,1)$-neighbour transitive codes corresponds to the \emph{$X$-neighbour transitive} codes, which have been studied by the first and fourth authors in \cite{ntrcodes,gillespiediadntc}, with a certain class characterised in \cite{gillespieCharNT}. 

Completely transitive codes form a sub-family of completely regular codes, introduced in 1973 by Delsarte \cite{delsarte1973algebraic}, and $(X,s)$-neighbour transitive codes are a sub-family of $s$-regular codes, for each $s$ (see Definition~\ref{regcode}). This means, in particular, that the set of minimum weight codewords of an $(X,s)$-neighbour transitive code forms a $q$-ary $s$-design (see Definition~\ref{desdef} and Lemma~\ref{design}), a fact which aids us in our classification. Completely regular codes have been studied extensively ever since Delsarte introduced them as a generalisation of perfect codes. In particular, Brouwer et al.~\cite{brouwer} and Neumaier \cite{Neumaier1992353} give results on these codes using the theory of distance regular graphs. Not only are these codes of interest to coding theorists as they possess a high degree of combinatorial symmetry, but, due to a result of Brouwer et al. \cite[p.353]{brouwer}, they are also the building blocks of certain types of distance regular graphs. Borges et al. have classified all linear completely regular codes that have \emph{covering radius $\rho=2$} and an \emph{antipodal} dual code \cite{borges2010q}, showing that these codes are extensions of linear completely regular codes with covering radius $\rho=1$.  They also classified this latter family of codes, and proved that these codes are in fact coset-completely transitive.

The main result of the paper is as follows. 

\begin{theorem}\label{main}
 Suppose $C$ is a code in $H(m,q)$, with $|C|\geq 2$ and $\delta\geq 5$. Then $C$ is entry-faithful $(X,2)$-neighbour transitive if and only if $C$ is equivalent to either:
 \begin{enumerate}
  \item a binary repetition code, and $X$ is isomorphic to a group listed in Table~\ref{repetitiongroups}, or
  \item the even weight subcode of the punctured Hadamard code of length 12 and $X\cong M_{11}$.
 \end{enumerate}
\end{theorem}

\begin{table}
\begin{center}
\begin{tabular}{ccl}
\hline\noalign{\smallskip}
 $X$ & $m$ & conditions\\
 \noalign{\smallskip}\hline\noalign{\smallskip}
 $\leq \AGaL_d(r)$ & $r^d$ & $X$ has a $2$-homogeneous index $2$ subgroup\\
 $S_m$ & $m$ &\\
 $M_{22}\rtimes Z_2$ & $22$ &\\
 $\unrhd \PSU_3(r)$ & $r^6+1$ & $X$ has an index $2$ subgroup\\
 $\unrhd \PSL_d(r)$ & $\frac{r^d-1}{r-1}$ & $X$ has an index $2$ subgroup\\
 \noalign{\smallskip}\hline
\end{tabular}
\caption{Groups $X$ such that the binary repetition code in $H(m,2)$ is entry-faithful $(X,2)$-neighbour transitive.}
\label{repetitiongroups}
\end{center}
\end{table}

The punctured Hadamard $12$ code $\P$ and the even weight subcode $\E$ of $\P$ are defined in Section~\ref{nontrivsec}. The Hadamard $12$ and punctured Hadamard $12$ codes were shown to be completely transitive by the first and fourth authors \cite{Gillespie20131394}. We prove in Lemma~\ref{punchadamard}, however, that $\E$ is not $(X,3)$-neighbour transitive, for any group $X$. We prove the following with regards to entry faithful completely transitive codes:

\begin{theorem}\label{comptrans}
 Let $C$ be a code in $H(m,q)$ with $|C|\geq 2$ and minimum distance $\delta\geq 5$. Then $C$ is entry-faithful $X$-completely transitive for some group $X$, if and only if $q=2$ and $C$ is equivalent to the binary repetition code. In this case $X\cong S_m$.
\end{theorem}

\begin{remark}
 In the unpublished work \cite{gillespie2012classification}, available on the ArXiv, the first, second, and fourth authors deal with the case where $C$ is entry-faithful $X$-completely transitive. They deduce that, if $|C|\geq 2$ and $\delta\geq 5$, $C$ must be the repetition code and $X\cong S_m$. Due to the fact that the methods used here are substantially more straight-forward, we have included some of the results from \cite{gillespie2012classification}. Thus, this paper is a complete account which largely replaces the results found in \cite{gillespie2012classification} (and we note that \cite{gillespie2012classification} will not be published).
\end{remark}

In Proposition~\ref{equivpunchad} we classify all $2$-regular codes in $H(11,2)$ such that $\delta\geq 5$ and $|C|\geq 2$. We find that such a code is either the repetition code, the punctured code of the Hadamard code of length 12, or the even weight subcode of the punctured Hadamard code. This can be seen as an extension of \cite[Theorem 1.1(b)]{Gillespie20131394}, which gives the punctured code of the Hadamard 12 code as the only completely regular code with $\delta=6$ in $H(11,2)$. 

Our study of $(X,2)$-neighbour transitive codes requires the identification of all finite $2$-transitive groups which do not contain the alternating group; however, do contain a $2$-homogeneous subgroup with a different socle. Proposition~\ref{2tr2hominc} gives a classification of all groups $G<H\leq S_m$ such that $G$ is $2$-transitive of degree $m$, $H$ is $2$-homogeneous of degree $m$, the socle of $G$ is not $A_m$ and the socles of $G$ and $H$ are not equal.

Section~\ref{prelim} introduces notation for Hamming graphs and their automorphism groups, as well as regular codes and designs. The proof of Theorem~\ref{main} hinges on Propositions~\ref{ihom} and~\ref{x12trans}, and Corollary~\ref{2transM}, which give two different $2$-transitive actions for $X$, and also tell us that the stabiliser $X_{\b 0}$ of the zero codeword, must have a $2$-homogeneous action on entries. This allows us to use the powerful results of \cite{devillers2011imprimitive} and \cite{praeger1990inclusion}. In Section~\ref{trivsec}, we deal with the cases that the socle of $X$, that is, the group generated by all minimal normal subgroups of $X$, either is the alternating group $A_m$ or is equal to the socle of $X_{\b 0}$. In Section~\ref{nontrivsec}, we deal with the case that the socle of $X$ is not $A_m$ and not equal to the socle of $X_{\b 0}$.  This allows us to complete the proof of the main results.

\section{Preliminaries}\label{prelim}

Let $M$ and $Q$ be sets of size $m$ and $q$ respectively, with $m,q\geq 2$. We refer to $M$ as the \emph{set of entries} and $Q$ as the \emph{alphabet}. The vertex set of the Hamming graph $\varGamma=H(m,q)$ consists of all $m$-tuples with entries labelled by the set $M$, taken from the set $Q$. An edge exists between two vertices if they differ as $m$-tuples in exactly one entry. For vertices $\alpha,\beta\in\varGamma$ the \textit{Hamming distance} $d(\alpha,\beta)$ (that is the distance in $\varGamma$) is the number of entries in which $\alpha$ and $\beta$ differ.

The \emph{minimum distance} of a code $C$ is $\delta=\min\{d(\alpha,\beta)\mid \alpha,\beta\in C,\alpha\neq \beta\}$. For $\alpha\in\varGamma$, define $\varGamma_r(\alpha)=\{\beta\in\varGamma \mid d(\alpha,\beta)=r\}$. Given $\alpha\in\varGamma$, define $d(\alpha,C)=\min\{d(\alpha,\beta) \mid \beta\in C\}$.  We then have the \textit{covering radius} $\rho =\max\{d(\alpha,C)\mid\alpha\in\varGamma\}$. For any $r\leq \rho$, define $C_r=\{\alpha\in\varGamma \mid d(\alpha,C)=r\}$. Note that $C_i$ is the disjoint union $\cup_{\alpha\in C}\varGamma_i(\alpha)$ for $i\leq \lfloor\frac{\delta-1}{2}\rfloor$. For $\alpha\in \varGamma$, we refer to the $i$-th entry of $\alpha$ as $\alpha_i$, so that $\alpha=(\alpha_1,\ldots,\alpha_m)$.

The automorphism group $\Aut(\varGamma)$ of the Hamming graph is the semi-direct product $N\rtimes L$, where $N\cong S_q^m$ and $L\cong S_m$ (see \cite[Theorem 9.2.1]{brouwer}). We refer to $L$ as the \emph{top group} of $\Aut(\varGamma)$. Let $g=(g_1,\dots,g_m)\in N$, $\sigma\in L$ and $\alpha\in\varGamma$. Then $g$ and $\sigma$ act on $\alpha$ as follows: 
\begin{equation*}
\alpha^g =(\alpha_1^{g_1},\ldots,\alpha_m^{g_m})\quad\text{and}\quad
\alpha^\sigma=(\alpha_{1{\sigma^{-1}}},\ldots,\alpha_{m{\sigma^{-1}}}).
\end{equation*}

The automorphism group of a code $C$ in $\varGamma$ is $\Aut(C)=\Aut(\varGamma)_C$, the setwise stabiliser of $C$ in $\Aut(\varGamma)$. Throughout the paper, the group $X$ is a subgroup of $\Aut(C)$.

The \emph{distance distribution of $C$} is the $(m+1)$-tuple $a(C)=(a_0,\ldots,a_m)$ where 
\begin{equation}\label{ai}
 a_i=\frac{|\{(\alpha,\beta)\in C\times C\,:\,d(\alpha,\beta)=i\}|}{|C|}.
\end{equation}   
We observe that $a_i\geq 0$ for all $i$ and $a_0=1$.  Moreover, $a_i=0$ for $1\leq i\leq \delta-1$ and $|C|=\sum_{i=0}^ma_i$. In the case where $q$ is a prime power, the \emph{MacWilliams transform} of $a(C)$ is the $(m+1)$-tuple $a'(C)=(a_0',\ldots,a_m')$ where 
\begin{equation}\label{kracheqn} 
 a'_k:=\sum_{i=0}^ma_iK_k(i)
\end{equation}
with 
\begin{equation*}
 K_k(x):=\sum_{j=0}^k(-1)^j\binom{x}{j}\binom{m-x}{k-j}(q-1)^{k-j}.
\end{equation*}
It follows from \cite[Lemma~5.3.3]{van1999introduction} that $a'_k\geq 0$ for $k=0,1,\ldots,m$. 

Let $0$ denote a distinguished element of the alphabet $Q$. For $\alpha\in V(\varGamma)$, the \emph{support of $\alpha$} is the set $\supp(\alpha)=\{i\in M\,:\,\alpha_i\neq 0\}$. The \emph{weight of $\alpha$} is defined as $\wt(\alpha)=|\supp(\alpha)|$. For any $a\in Q\backslash\{0\}$ we use the notation $(a^k,0^{m-k})$ to denote the vertex in $V(\varGamma)$ that has $a$ in the first $k$ entries, and $0$ in the remaining entries, and if $k=0$ we denote the vertex by ${\bf{0}}$.  

\begin{lemma}\label{supplem}
 Let $\alpha={\bf{0}}$ and $x=(g_1,\ldots,g_m)\sigma\in\Aut(\varGamma)_\alpha$. Then $\supp(\beta^x)=\supp(\beta)^\sigma$ for all $\beta\in V(\varGamma)$.     
\end{lemma}

\begin{proof}
 Since each $g_i$ fixes $0$ and $\sigma$ permutes coordinates, the $i$th entry of $\beta^x$ is non-zero if and only if the $i^{\sigma^{-1}}$ entry of $\beta$ is non-zero.  Thus the result follows.  
\end{proof}

For $\alpha=(\alpha_i)$, $\beta=(\beta_i)\in V(\varGamma)$, we let $\Diff(\alpha,\beta)=\{i\in M\,:\,\alpha_i\neq\beta_i\}.$ Now suppose $|C|\geq 2$ and $\alpha,\beta\in C$.  Then we let $$\Diff(\alpha,\beta,C)=\{\gamma\in C\,:\,\Diff(\alpha,\gamma)=\Diff(\alpha,\beta)\}.$$ 
By definition, $\beta\in\Diff(\alpha,\beta,C)$, so $\Diff(\alpha,\beta,C)\neq\emptyset$.

\begin{lemma}\label{favvertex}
 Let $C$ be a code with minimum distance $\delta$ and $|C|\geq 2$, and let $\alpha,\beta\in C$ such that $d(\alpha,\beta)=\delta$. Then for all $a\in Q$, there exists $x\in\Aut(\varGamma)$ such that the following two conditions hold.
\begin{itemize}
 \item[(i)] $\alpha^x=(a,\ldots,a)$, and
 \item[(ii)] for each $\gamma\in\Diff(\alpha,\beta,C)$, $\gamma^x=(c^{\delta},a^{m-\delta})$ for some $c\in  Q\backslash\{a\}$.    
\end{itemize}
\end{lemma}

\begin{proof}
 Let $\Diff(\alpha,\beta,C)=\{\beta^1,\ldots,\beta^s\}$. It follows that $(\beta^i)_k=\alpha_k$ for each $i\leq s$ and $k\in  M\backslash\Diff(\alpha,\beta)$.  Therefore, because $C$ has minimum distance $\delta$, $d(\beta^i,\beta^j)=\delta$ for each distinct pair $\beta^i,\beta^j\in\Diff(\alpha,\beta,C)$.  This implies that for each $k\in\Diff(\alpha,\beta)$, the $s+1$ entries $\alpha_k,(\beta^1)_k,\ldots,(\beta^s)_k$ are pairwise distinct elements of $Q$.  Thus $s\leq q-1$.  Let $a\in Q$ and $\{c_1,\ldots,c_s\}\subseteq Q\backslash\{a\}$.  Since $S_q$ acts $q$-transitively on $Q$, it follows that for each $k\in \Diff(\alpha,\beta)$ there exists $h_k\in S_q$ such that $(\beta^i)_k^{h_k}=c_i$ for each $i\leq s$ and $\alpha_k^{h_k}=a$.  Also for each $k\in M\backslash\Diff(\alpha,\beta)$ let $h_k=(a\,\,\alpha_k)\in S_q$. Now let $h=(h_1,\ldots,h_m)\in S_q^m$.  Since $S_m$ acts $m$-transitively on $M$ and $|\Diff(\alpha,\beta)|=\delta\leq m$, there exists $\sigma\in S_m$ such that $\Diff(\alpha,\beta)^\sigma=\{1,\ldots,\delta\}$.  Let $x=h\sigma\in\Aut(\varGamma)$.  Then $\alpha^x=(a,\ldots,a)$ and $(\beta^i)^x=(c_i^\delta,a^{m-\delta})$ for each $i=1,\ldots,s$.      
\end{proof}

We say that two codes, $C$ and $C'$, in $H(m,q)$, are \textit{equivalent} if there exists $x\in \Aut(\varGamma)$ such that $C^x=C'$. Since elements of $\Aut(\varGamma)$ preserve distance, equivalence preserves minimum distance. 

Finally, we denote the \emph{set of $k$-subsets} of a set $\Omega$ by $\Omega^{\{k\}}$. For a set $\Omega$ and group $G\leq\Sym(\Omega)$, we say $G$ acts \emph{$k$-homogeneously on $\Omega$} if $G$ acts transitively on $\Omega^{\{k\}}$. 

\subsection{$s$-Neighbour Transitive Codes}\label{sntr}

\begin{definition}
 Let $C$ be a code in $H(m,q)$ and $X\leq\Aut(C)$. If $X$ is transitive on $C,C_1,\ldots, C_s$, then we say $C$ is an \emph{$(X,s)$-neighbour transitive} code, or simply an \emph{$s$-neighbour transitive} code, if $X$ is clear from the context. Moreover, $X$-completely transitive codes (defined in Section \ref{intro}) are the $(X,\rho)$-neighbour transitive codes.
\end{definition}

\begin{remark}\label{remfav}
 Let $x\in\Aut(\varGamma)$, and let $C$ be an $(X,s)$-neighbour transitive code with minimum distance $\delta$.  By an approach similar to that used in \cite[Lemma~2]{gillespiediadntc}, it holds that $C^x$ is $(X^x,s)$-neighbour transitive, and because minimum distance is preserved by equivalence, $C^x$ has minimum distance $\delta$.  Thus for any $a\in Q\backslash\{0\}$, Lemma \ref{favvertex} allows us to replace $C$ with an equivalent $(X,s)$-neighbour transitive code with minimum distance $\delta$ that contains both ${\bf{0}}$ and $(a^\delta,0^{m-\delta})$.          
\end{remark}

There are two additional actions of $X\leq \Aut(\varGamma)$ which will be useful to us. First we consider the action of $X$ on the set of entries $M$, which we will write as $X^M$, defined by the following homomorphism:

\begin{center}
\begin{tabular}{cccc}
 $\mu$ :& $X$ &$\longrightarrow$ & $S_m$\\
& $(h_1,\ldots,h_m)\sigma$ &$\longmapsto$ & $\sigma$ 
\end{tabular}
\end{center}

\begin{proposition}\label{ihom} 
 Let $C$ be an $(X,s)$-neighbour transitive code with minimum distance $\delta$. Then for $\alpha\in C$ and $i\leq\min\{s,\lfloor\frac{\delta-1}{2}\rfloor\}$, the stabiliser $X_\alpha$ fixes setwise and acts transitively on $\varGamma_i(\alpha)$.  In particular, $X_\alpha$ acts $i$-homogeneously on $M$.    
\end{proposition}

\begin{proof} 
 By replacing $C$ with an equivalent code if necessary, Remark~\ref{remfav} allows us to assume that $\alpha={\bf{0}}\in C$.  Firstly, because automorphisms of the Hamming graph preserve distance, it follows that $X_\alpha\leq X_{\varGamma_i(\alpha)}$.  Now let $\nu_1,\nu_2\in\varGamma_i(\alpha)$.  As $C_i$ is an $X$-orbit, and because $\varGamma_i(\alpha)\subseteq C_i$, there exists $x\in X$ such that $\nu_1^x=\nu_2$. Suppose $x\notin X_\alpha$.  Then $\alpha\neq \alpha^x\in C$, and so $d(\alpha,\alpha^x)\geq \delta$.  However, $d(\alpha,\alpha^x)\leq 2i<\delta$, which is a contradiction.  Thus $X_\alpha$ acts transitively on $\varGamma_i(\alpha)$.  

 Finally, let $J_1$, $J_2\in M^{\{i\}}$, and $\nu, \gamma\in V(\varGamma)$ such that $\supp(\nu)=J_1$ and $\supp(\gamma)=J_2$. It follows that $\nu,\gamma\in\varGamma_i(\alpha)\subseteq C_i$.  As $X_\alpha$ acts transitively on $\varGamma_i(\alpha)$, there exists $x=(g_1,\ldots,g_m)\sigma\in X_\alpha$ such that $\nu^x=\gamma$.  A consequence of Lemma \ref{supplem} is that
  $J_1^\sigma=\supp(\nu)^\sigma=\supp(\nu^x)=\supp(\gamma)=J_2$. Hence $X_\alpha$ acts $i$-homogeneously on $M$. 
\end{proof}

\begin{corollary}\label{multtrans}
 Let $C$ be an $(X,s)$-neighbour transitive code with  minimum distance $\delta$.  Then for each $i\leq\min\{s,\lfloor\frac{\delta-1}{2}\rfloor\}$ and $I\in M^{\{i\}}$, the setwise stabiliser $X_I$ acts transitively on $C$. 
\end{corollary}

\begin{proof} 
 By definition $C$ is $(X,i)$-neighbour transitive and, by Proposition \ref{ihom}, $X_\alpha$ acts transitively on the set $M^{\{i\}}$ of $i$-subsets of $M$. Hence $X$ is transitive on $C\times M^{\{i\}}$, and so $X_I$ is transitive on $C$. 
\end{proof}

Next we consider the action of the subgroup $X_i\leq X$, stabilising the entry $i\in M$, on the alphabet $Q$. We denote this action by $X_i^Q$ and it is defined by the homomorphism:

\begin{center}
\begin{tabular}{cccc}
 $\varphi_i$ :& $X_i$ &$\longrightarrow$ & $S_q$\\
& $(h_1,\ldots,h_m)\sigma$ &$\longmapsto$ & $h_i$ 
\end{tabular}
\end{center}

\begin{proposition}\label{x12trans} 
 Let $C$ be an $(X,1)$-neighbour transitive code in $H(m,q)$ with $\delta\geq 3$ and $|C|>1$.  Then $X_1^{Q}$ acts $2$-transitively on $Q$.
\end{proposition}

\begin{proof} 
 Let $a\in Q\backslash\{0\}$.  By replacing $C$ with an equivalent code if necessary, Remark \ref{remfav} allows us to assume that $\alpha={\bf{0}}$ and $\beta=(a^\delta,0^{m-\delta})$ are two codewords of $C$.  Choose any $b\in Q\backslash\{0\}$.  As $\delta\geq 3$ it follows that $\nu_1=(a,0^{m-1})$, $\nu_2=(b,0^{m-1})\in\varGamma_1(\alpha)\subseteq C_1$.  By Proposition~\ref{ihom}, there exists $x=(g_1,\ldots,g_m)\sigma\in X_\alpha$ such that $\nu_1^x=\nu_2$.  Consequently, Lemma \ref{supplem} implies that $1^\sigma=1$.  Thus $a^{g_1}=b$, and because $x\in X_\alpha$, we conclude that $g_1\in(X_1^{Q})_0$, the stabiliser of $0$ in $X_1^Q$.  Hence $(X_1^{Q})_0$ acts transitively on $Q\backslash\{0\}$. By Corollary~\ref{multtrans}, $X_1$ acts transitively on $C$.  Hence there exists $y=(h_1,\ldots,h_m)\pi\in X_1$ such that $\alpha^y=\beta$.  As $y\in X_1$ we have that $0^{h_1}=a$ and $h_1\in X_1^Q$.  Thus $X_1^Q$ acts $2$-transitively on $Q$. 
\end{proof}

\begin{definition}
 Let $C$ be an $(X,s)$-neighbour transitive code in $H(m,q)$. If $X^M$ is faithful, in other words $X\cap N\cong S_q^m=1$, then we say $C$ is an \emph{entry-faithful} $(X,s)$-neighbour transitive code.
\end{definition}

\begin{corollary}\label{2transM} 
 Let $C$ be an entry-faithful $(X,2)$-neighbour transitive code, with $|C|\geq 2$ and $\delta\geq 5$. Then $X$ acts $2$-transitively on $M$.  
\end{corollary}

\begin{proof} 
 By Proposition \ref{ihom}, $X_\alpha$ acts $2$-homogeneously on $M$, and so $X$ has a faithful $2$-homogeneous action on $M$.  By Proposition~\ref{x12trans}, $X_1^Q$ acts $2$-transitively on $Q$.  Thus $X_1^Q$ has even order, and so $X$ has even order.  Therefore, by \cite[Lemma~2.1]{sharp}, $X$ acts $2$-transitively on $M$.       
\end{proof}

The following result gives us restrictions on the orders of the groups $X$ and $X_{\b 0}$.

\begin{lemma}\label{codeorder}
 If $C$ is an $(X,2)$-neighbour transitive code with $\delta\geq 5$ and $\b 0\in C$, then $\binom{m}{2} (q-1)^2$ divides $|X_{\b 0}|$, and hence $|X|$. In particular, if $|X_{\b 0}|=m(m-1)/2$ then $q=2$.
\end{lemma}

\begin{proof}
 Since $\delta\geq 5$ we have that $X_{\b 0}$ is transitive on $\varGamma_2(\b 0)$. Thus $|\varGamma_2(\b 0)|=\binom{m}{2} (q-1)^2$ divides $|X_{\b 0}|$. 
\end{proof}

\subsection{Regular Codes and Designs}\label{regular}

\begin{definition}\label{regcode}
 Let $C$ be a code with covering radius $\rho$ and $s$ be an integer with $0\leq s\leq \rho$. Then $C$ is \emph{$s$-regular} if for each $i$, $0\leq i \leq s$, each vertex $\nu\in C_i$, and each $k$, $0\leq k \leq m$, the number $|\varGamma_k(\nu)\cap C|$ depends only on $i$ and $k$. If $s=\rho$ then we say $C$ is \emph{completely regular}.
\end{definition}

Let $\alpha\in H(m,q)$ and $0\in Q$. The vertex $\nu$ is said to be \emph{covered} by $\alpha$, if for every $i\in M$ such that $\nu_i\neq 0$ we have $\nu_i=\alpha_i$. 

\begin{definition}\label{desdef}
 A \emph{$q$-ary $s$-$(v,k,\lambda)$ design} is a subset $\mathcal{D}$ of vertices of $\varGamma_k(\b 0)$ (where $k\geq s$) such that each vertex $\nu \in\varGamma_s(\b 0)$ is covered by exactly $\lambda$ vertices of $\mathcal{D}$. When $q=2$, $\D$ is simply the set of characteristic vectors of an $s$-design. We refer to the elements of $\D$ as \emph{blocks}.
\end{definition}

We will use the following equations, which can be found for instance in \cite{stinson2004combinatorial}. Let $\D$ be a binary $s$-$(v,k,\lambda)$ design with $|\D|=b$ blocks and let $r$ be the number of blocks incident with a point. We then have $vr=bk$, $r(k-1)=\lambda(v-1)$ and $$b=\frac{v(v-1)\cdots(v-s+1)}{k(k-1)\cdots(k-s+1)}\lambda.$$

The following is \cite[Theorem~2.4.7]{van1976uniformly}. Note that if there are no codewords of a given weight then we have the empty design, with $\lambda=0$.

\begin{theorem}\label{regdes}
 Let $C$ be an $s$-regular code such that $\b 0\in C$ and $\delta\geq 2s$. Then for each $k$ such that $0\leq k\leq m$, the set of codewords from $C$ of weight $k$ forms a $q$-ary $s$-$(m,k,\lambda)$ design, for some $\lambda$.
\end{theorem}

\begin{definition}\label{repcode}
 The \emph{repetition code in $H(m,q)$}, denoted by $\Rep(m,q)$, is equal to the set of vertices of the form $(a,\ldots,a)$, for all $a\in Q$.  It has minimum distance $\delta=m$.
\end{definition}  

The binary repetition code is one of the two entry-faithful $2$-neighbour transitive codes in Theorem~\ref{main}. The next result shows that it is in fact the only $2$-regular code such that $4\leq \delta=m$.

\begin{lemma}\label{sizeofcode}
 Let $C$ be a code with $|C|\geq 2$ and $\delta=m$. Then there exists $C'$ equivalent to $C$ with $C'\subseteq\Rep(m,q)$.  Moreover if $C$ is $1$-regular then $C'=\Rep(m,q)$; if $C$ is $2$-regular and $m\geq 4$ then $C'=\Rep(m,2)$.
\end{lemma}

\begin{proof}
 Let $0,a\in Q$. Since $\delta=m$, Lemma \ref{favvertex} implies that there is a code $C'$ equivalent to $C$ which contains $\alpha=(0,\ldots,0)$ and $\beta=(a,\ldots,a)$, and each $\gamma\in C'\backslash\{\alpha,\beta\}$ at distance $\delta=m$ from $\alpha$ is of the form $(b,\ldots,b)$ for some $b\in Q\backslash\{0,a\}$. Thus, $C'$ is a subset of the repetition code $\Rep(m,q)$, and in particular $|C|=|C'|\leq q$.    
 
 Assume $C$, and hence $C'$, is $1$-regular. Suppose $|C'|<q$.  Then there exists $b\in Q\backslash\{0,a\}$ such that $b$ does not appear in any codeword of $C'$.  Let $\nu_1=(a,0^{m-1})$ and $\nu_2=(b,0^{m-1})$.  Then $\nu_1,\nu_2\in C'_1$.  However, $|\varGamma_{m-1}(\nu_1)\cap C'|=2$ if $m=2$, and $1$ if $m\geq 3$, while $|\varGamma_{m-1}(\nu_2)\cap C'|=1$ if $m=2$ and $0$ if $m\geq 3$, which is a contradiction.  Therefore $|C'|=q$ and $C'=\Rep(m,q)$.      

 Now assume that $C$, and hence $C'$, is $2$-regular, and that $m\geq 4$. Let $\nu_3=(a,a,0^{m-2})$.  As $\nu_3$ has weight $2$ and $\delta=m\geq 4$, we have $\nu_3\in C'_2$. Let $\gamma=(c,\ldots,c)$ be an arbitrary element of $C'$. Then 
 \[d(\nu_3,\gamma)=\left\{\begin{array}{cl} 2&\textnormal{if
  $c=0$,}\\ m-2&\textnormal{if 
 $c=a$,}\\
 m&\textnormal{if $c\in
 Q\backslash\{0,a\}$.}\end{array}\right.\]  
 Since $m\geq 4$ it follows that $\varGamma_{m-1}(\nu_3)\cap C=\emptyset$. Therefore, because $C'$ is $2$-regular, $\varGamma_{m-1}(\nu)\cap C'=\emptyset$ for all $\nu\in C'_2$.  Now suppose that $q\geq 3$ and consider $\nu_4=(b,a,0^{m-2})\in C'_2$, where $b\neq 0,a$. Then $d(\nu_4,\beta)=m-1$, which is a contradiction. Therefore $q=2$, and we note that $\Rep(m,2)$ is $2$-regular. 
\end{proof}

The following shows that an $(X,s)$-neighbour transitive code is $s$-regular, which allows us to use the above results.

\begin{lemma}\label{design}
 Let $C$ be an $(X,s)$-neighbour transitive code. Then $C$ is $s$-regular. Furthermore, if $\b 0\in C$ and $\delta\geq 2s$ then the set of codewords of weight $k\leq m$ forms a $q$-ary $s$-$(m,k,\lambda)$ design, for some $\lambda$.
\end{lemma}

\begin{proof}
 Let $i$ be such that $0\leq i\leq s$ and $\nu,\mu\in C_i$. Then there exists an $x\in X$ such that $\nu^x=\mu$. It follows that $|\varGamma_k(\nu)\cap C|=|\varGamma_k(\mu)\cap C|$ for all $k$ with $0\leq k\leq m$ and $C$ is $s$-regular. If $\b 0\in C$ we can apply Theorem~\ref{regdes}. 
\end{proof}

\section{The Socles of $X$ and $X_{\b 0}$}\label{trivsec}

In Proposition~\ref{x12trans}, Proposition~\ref{ihom} and Corollary~\ref{2transM} we saw that an $(X,2)$-neighbour transitive code has two different $2$-transitive actions of $X$, and that the stabiliser $X_{\b 0}$ of the zero codeword is $2$-homogeneous on $M$. A $2$-homogeneous group $G$ is either \emph{affine} or \emph{almost-simple}, depending on the structure of the \emph{socle} of $G$, denoted $\soc(G)$, that is, the group generated by all the minimal normal subgroups of $G$ (see, for example, \cite{dixon1996permutation}). The group is affine if the socle is a regular elementary abelian group, and the group is almost-simple if the socle is a non-abelian simple group.

The following remark, which relies directly on \cite[Theorem 1.4]{devillers2011imprimitive}, will allow us to greatly reduce the possibilities for the parameters of an entry-faithful $(X,2)$-neighbour transitive code $C$ of a given group $X$.

\begin{remark}\label{imprimrank3}
 Let $C$ be an entry-faithful $(X,2)$-neighbour transitive code in the Hamming graph $H(m,q)$. Then, by Proposition~\ref{x12trans}, $X_1$ acts $2$-transitively on $Q$ in the first entry and, by Corollary~\ref{2transM}, $X$ acts $2$-transitively on $M$. It follows that $X_1^Q(\leq \Sym(Q))$, $X^M(\leq S_m)$ and $X$ satisfy \cite[Hypothesis 1]{devillers2011imprimitive} ($G$ there being our $X$). That is, $X_1^Q$ and $X^M$ are $2$-transitive, $X\leq X_1^Q\wr X^M$ acting imprimitively on $Q\times M$, with $X$ projecting onto $X^M$, and $X_1^Q$ being, what is referred to there as, the component of $X$. Thus, since $X$ is almost simple and acts faithfully on $M$, we can apply \cite[Theorem 1.4]{devillers2011imprimitive}, so that the possible values of $X$, $m$, and $q$ are listed, as $G$, $n$, and $|B|$, in one of the lines of \cite[Tables 2 or 3]{devillers2011imprimitive}.
\end{remark}

The socle of $X_{\b 0}$ may or may not be the same as the socle of $X$. In the case that they are the same, the next lemma shows that the size of an $(X,2)$-neighbour transitive code is at most $2$.

\begin{lemma}\label{equalsocles}
 Let $C$ be an entry-faithful $(X,2)$-neighbour transitive code in $H(m,q)$ such that $|C|\geq 2$, $\delta\geq 5$, and $\soc(X)=\soc(X_{\b 0})$. Then $C$ is the binary repetition code and $X$ is isomorphic to one of the groups in Table~\ref{repetitiongroups}.
\end{lemma}

\begin{proof}
 By Remark~\ref{remfav}, we can assume $\b 0\in C$. By Proposition~\ref{ihom}, $X$ acts $2$-homogeneously on $M$, and thus $X$ acts primitively on $M$. Since $H=\soc(X_{\b 0})$ is a normal subgroup of $X$, it follows that $H$ acts transitively on $M$. Moreover, as $X$ is transitive on $C$, the stabiliser of each codeword is conjugate in $X$. Now, $H$ is normal in $X$, and thus $H$ is contained in the stabiliser of every codeword from $C$. Let $\beta\in C$ such that $d(\b 0,\beta)=\delta$. Without loss of generality, by Remark~\ref{remfav}, $\beta=(b^\delta,0^{m-\delta})$. Let $x=(h_1,\ldots,h_m)\sigma\in H$. Then $x$ fixes $\b 0$, and so $h_i$ fixes $0$ for all $i$. Since $x$ also fixes $\beta$, it follows that $\sigma$ fixes each of the two subsets of entries $\{1,\ldots,\delta\}$ and $\{\delta+1,\ldots,m\}$ setwise. This is true for each $x\in H$, so $H$ fixes $\{1,\ldots,\delta\}$ setwise. Since $H$ acts transitively on $M$, it follows that $\delta=m$. Finally, we apply Lemma~\ref{sizeofcode}, since any $2$-neighbour transitive code is also $2$-regular, and conclude that $q=2$ and $C$ is equivalent to the binary repetition code.
 
 Now $X^M$ is $2$-transitive, and is thus affine or almost simple. Suppose $X$ is affine. In this case, since $|X:X_{\b 0}|=|C|=2$, it follows that $X$ is any $2$-transitive subgroup of $\AGL_d(q)$, which contains an index two subgroup $X_{\b 0}$ acting $2$-homogeneously on $M$, with $\soc(X)=\soc(X_{\b 0})$. Suppose, then, that $X$ is almost simple. By Remark~\ref{imprimrank3}, $X$, $m$, and $q$ appear (as $G$, $n$, and $|B|$ respectively) in \cite[Tables 2 and 3]{devillers2011imprimitive}. We also require $X$ to have a $2$-homogeneous (and hence, by \cite{kantor1972k}, $2$-transitive) index two subgroup $X_{\b 0}$, with $\soc(X)=\soc(X_{\b 0})$. All possibilities for $X$ are then contained in Table~\ref{repetitiongroups}. 
\end{proof}

The next result deals with the possibility that $X\cong A_m$ or $S_m$.

\begin{lemma}\label{amsm}
 Let $C$ be an $(X,2)$-neighbour transitive code in $H(m,q)$, such that $\delta\geq 5$ and $|C|\geq 2$, where $X$ is entry-faithful and $X\cong A_m$ or $S_m$. Then $C$ is equivalent to the binary repetition code and $X\cong S_m$.
\end{lemma}

\begin{table}
\begin{center}
\begin{tabular}{cccccc}
\hline\noalign{\smallskip}
 $X$ & $m$ & $q$ & $|X|$ & $\binom{m}{2} (q-1)^2$ & Lemma~\ref{codeorder} \\
 \noalign{\smallskip}\hline\noalign{\smallskip}
 $A_m,S_m$ & $m$ & $m-1$ & & & \\
 $S_m$ & $m$ & $2$ & & & \\
 $S_{5}$ & $5$ & $3$ & $2^3.3.5$ & $2^3.5$ & \\
 $A_{6}$ & $6$ & $6$ & $2^3.3^2.5$ & $3.\b{5^3}$ & fail \\
 $S_{6}$ & $6$ & $6$ & $2^4.3^2.5$ & $3.\b{5^3}$ & fail \\
 $A_{7}$ & $7$ & $10$ & $2^3.3^2.5.7$ & $\b{3^5}.7$ & fail \\
 $S_{7}$ & $7$ & $10$ & $2^4.3^2.5.7$ & $\b{3^5}.7$ & fail \\
 $A_{8}$ & $8$ & $15$ & $2^6.3^2.5.7$ & $2^4.\b{7^3}$ & fail \\
 $A_{9}$ & $9$ & $15$ & $2^6.3^4.5.7$ & $2^4.3^2.\b{7^2}$ & fail \\
 \noalign{\smallskip}\hline
\end{tabular}
\caption{Candidate entry-faithful groups $X$, with $soc(X)=A_m$, such that $C$ is an $(X,2)$-neighbour transitive code in $H(m,q)$.}\label{socleAm}
\end{center}
\end{table}

\begin{proof}
 By Remark~\ref{imprimrank3} the values of $X$, $m$, and $q$ appear as $G$, $n$, and $|B|$ respectively, in \cite[Table 2]{devillers2011imprimitive}. Such groups, with the additional requirement that $\soc(X)=A_m$, are compiled in Table~\ref{socleAm}. Those cases which fail to satisfy Lemma~\ref{codeorder} are marked as such. Thus, we are left with the cases $q=m-1$; $X\cong S_m$ and $q=2$; and $X \cong S_5$, $m=5$ and $q=3$. Note that $m\geq 5$, since $\delta\geq 5$, and $X_{\b 0}$ is $2$-homogeneous on $M$, by Proposition~\ref{ihom}. 

 First, let $m=5$, $q=3$ and $X\cong S_5$. Table~\ref{socleAm} and Lemma~\ref{codeorder} imply that $X_{\b 0}$ is isomorphic to an index $3$ subgroup of $S_5$, which does not exist.
 
 Now, let $X\cong S_m$ and $q=2$. If $X_{\b 0}\cong A_m$ then, by Lemma~\ref{equalsocles}, $C$ is the binary repetition code, by Lemma~\ref{equalsocles}. Suppose $X_{\b 0}\neq A_m$. Now, $X_{\b0}$ is $2$-homogeneous, and hence primitive. Thus we can then apply a result from \cite{bochert1889ueber}, which says that the index of a primitive group, not containing $A_m$, in $S_m$ is at least $\lfloor(m+1)/2\rfloor!$. Now, $|C|=|X:X_{\b 0}|\geq \lfloor(m+1)/2\rfloor!$. By the Singleton bound $|C|\leq 2^{m-\delta+1}\leq 2^{m-4}$. Hence $m$ must satisfy $2^{m-4}\geq \lfloor(m+1)/2\rfloor!$, which does not hold for $m\geq 5$.
 
 Suppose $X\cong A_m$ or $S_m$ and $q=m-1$. If $X_{\b 0}$ is $2$-homogeneous, but not $2$-transitive, then $X_{\b 0}\leq \AGaL_1(m)$ and $m\equiv 3\pmod{4}$, by \cite{kantor1972k}. So $m$ is a prime power, that is $m=r^t$, and $|X_{\b 0}|$ divides $tm(m-1)/2$. However, by Lemma~\ref{codeorder}, $\binom{m}{2}(q-1)^2=m(m-1)(m-2)^2/2$ must divide $|X_{\b 0}|$, that is, $(m-2)^2$ must divide $t$. This does not occur for $m\geq 5$, so $X_{\b 0}$ is $2$-transitive.  
 
 Suppose $m=7$. We have that $|X|$ divides $7!$. However, $(q-1)^2\binom{m}{2}=3.5^2.7$ must divide $|X|$, by Lemma~\ref{codeorder}, which does not hold. 
 
 Suppose $m\neq 7$. If $X=A_m$ or $S_m$, then the stabiliser $X_1$ of $1\in M$ is $A_{m-1}$ or $S_{m-1}$, respectively. Now $X_1$ is transitive of degree $m-1$ on $Q$ in the first entry, by Proposition~\ref{x12trans}, and transitive of degree $m-1$ on $M\setminus\{1\}$, since $X$ is $2$-transitive on $M$, by Corollary~\ref{2transM}. Let $x\in Q\setminus \{0\}$ . Since $A_{m-1}$ and $S_{m-1}$ only have one transitive representation of degree $m-1$, up to permutational isomorphism, it follows that $X_{1,x}=X_{1,i}$, for some $i\in M$.
 
 Now, consider the stabiliser $X_{{\b 0},1}$ of the zero codeword and the entry $1\in M$. Since $X_{\b 0}$ is transitive on the neighbours of $\b 0$, it follows that $X_{{\b 0},1}$ is transitive on the vertices of the form $\nu=(y,0^{m-1})$ in $\varGamma_1(\b 0)$, where $y\in Q\setminus \{0\}$. Hence, $X_{\b 0,1}$ is transitive on $Q\setminus \{0\}$. Thus, for $x\in Q\setminus \{0\}$ the stabiliser $X_{\b 0,1,x}$, has index $m-2$ in $X_{\b 0,1}$. Note that $X_{{\b 0},1,x}=(X_{1,x})_{\b 0}=X_{1,i,\b 0}$, and we have just proved that this subgroup has index $m(m-2)$ in $X_{\b 0}$. Since $X_{\b 0}$ is $2$-transitive on $M$, it follows that $X_{\b 0,1,i}$ has index $m(m-1)$ in $X_{\b 0}$, which is a contradiction. 
\end{proof}

\section{Distinct Socles for $X,X_{\b 0}$}\label{nontrivsec}

We begin with an example of an entry-faithful $(X,2)$-neighbour transitive code, where $\soc(X)\neq \soc(X_{\b 0})$.

\begin{definition}\label{hadamarddef}
 Let $\mathcal{P}$ be the punctured Hadamard $12$ code, obtained as follows (see \cite[Part 1, Section 2.3]{macwilliams1978theory}). First, we construct a normalised Hadamard matrix $H_{12}$ of order $12$ using the Paley construction. Let $M=\F_{11}\cup \{*\}$ and let $H_{12}$ be the $12\times 12$ matrix with first row $v$, where $v_a=-1$ if $a$ is a square in $\F_{11}$ (including $0$), and $v_a=1$ if $a$ is a non-square or if $a=*$, taking the orbit of $v$ under the additive group of $\F_{11}$ acting on $M$ to form $10$ more rows and adding a final row, the vector $((-1)^{12})$. The Hadamard code $\H$ of length $12$ in $H(12,2)$ then consists of the vertices $\alpha$ such that, for $u$ a row in $H_{12}$ or $-H_{12}$, $\alpha_a=0$ if $v_a=1$ and $\alpha_a=1$ if $v_a=-1$. The punctured code $\P$ of $\H$ is obtained by deleting the coordinate $*$ from $M$. The weight $6$ codewords of $\P$ form a binary $2$-$(11,6,3)$ design, which we denote throughout by $\D$. The code $\P$ consists of the following codewords: the all zero vector, the all ones vector, the characteristic vectors of the $2$-$(11,6,3)$ design $\D$, and the characteristic vectors of the complement of that design, which forms a $2$-$(11,5,2)$ design. The even weight subcode $\mathcal E$ of $\mathcal P$ is the code consisting of the all zero vector and the $2$-$(11,6,3)$ design. Both $\D$ and its complement are unique up to isomorphism \cite{todd1933}.
\end{definition}

\begin{lemma}\label{punchadamard}
 Let $\mathcal E$ be the even weight subcode of the punctured Hadamard $12$ code, defined above. Then $\mathcal E$ is entry-faithful $(X,2)$-neighbour transitive, if and only if $X=\Aut(\mathcal E)\cong M_{11}$. Moreover $\Aut(\E)$ acts faithfully on entries. Furthermore, $\E$ is not $(X,3)$-neighbour transitive for any group $X$, in particular, $\E$ is not completely transitive.
\end{lemma}

\begin{proof}
 By \cite{Gillespie20131394} the Hadamard $12$ code $\mathcal{H}$ has automorphism group the non-split extension $2 M_{12}$ and $\mu(\Aut(\mathcal H))=M_{12}$ is $5$-transitive on the entries of $H(12,2)$ \cite{hallmathieu12}. Since $\mathcal P$ is obtained by deleting an entry from $\mathcal H$, and since the Schur multiplier of $M_{11}$ is trivial, we have that $\Aut(\mathcal P)= H\times M_{11}\cong 2 \times M_{11}$, where $H=\langle ((01),\ldots,(01))\rangle$. Let $Y$ be the stabiliser in $\Aut(\varGamma)$ of the set of even weight vertices of $H(11,2)$, noting that $|\Aut(\varGamma):Y|=2$. Now, either $\Aut(\P)\leq Y$ or $Y\cap \Aut(\P)$ is an index $2$ subgroup of $\Aut(\P)$. Since $\P$ contains the weight $11$ codeword $(1^{11})$ and $\Aut(\P)$ is transitive on $\P$, it follows that $Y\cap \Aut(\P)$ must be an index $2$ subgroup of $\Aut(\P)$. As $M_{11}$ is the unique index $2$ subgroup of $\Aut(\P)$ we have $M_{11}=Y\cap\Aut(\P)$. Since $\Aut(\P)$ is transitive on $\P$, $M_{11}$ has two orbits on $\P$; the even weight codewords $\E$ and the odd weight codewords $\P\setminus\E$. Hence $M_{11}\leq \Aut(\E)$ and $M_{11}$ is transitive on $\E$. By \cite{todd1933}, $\PSL_2(11)$ is the automorphism group of the weight $6$ codewords forming the $2$-$(11,6,3)$ design $\D$. Thus $\Aut(\E)_{\b 0}\leq \PSL_2(11)$. Now, $|\Aut(\E)|=12.|\Aut(\E)_{\b 0}|\leq 12.|\PSL_2(11)|$. Hence $\Aut(\E)\cong M_{11}$ and $\Aut(\E)_{\b 0}\cong \PSL_2(11)$. Moreover, $M_{11}$ acts faithfully on entries since it contains no non-trivial normal subgroups and $\Aut(\E)^M$ is non-trivial, since it contains $\PSL_2(11)$.
 
 Now, $M_{11}$ is transitive on $\E$. Moreover, $\Aut(\E)_{\b 0}\cong \PSL_2(11)$ acts $2$-transitively on entries. Thus, it follows that $\Aut(\mathcal E)_{\b 0}$ is transitive on the weight $1$ and weight $2$ vertices of $H(11,2)$. Hence, $\E$ is $(\Aut(\E),2)$-neighbour transitive.
 
 Suppose $X$ is a proper subgroup of $\Aut(\E)$ such that $\E$ is $(X,2)$-neighbour transitive. Since $X$ is transitive on $\E$, it follows that $X_{\b 0}$ is a proper subgroup of $\PSL_2(11)$. Then, as $X_{\b 0}$ is transitive on $\varGamma_2(\b 0)$, we have that $X_{\b 0}$ is $2$-homogeneous, and hence primitive on $M$. The only $2$-homogeneous proper subgroup of $\PSL_2(11)$ in its action of $11$ points is $F_{55}$, and thus $X_{\b 0}\cong F_{55}$. Also $X^M$ is $2$-transitive on $11$ points, with order $|C|.|X_{\b 0}|=12.55$, implying $X\cong \PSL_2(11)$. However, $\Aut(\E)$ has a unique conjugacy class of subgroups isomorphic to $\PSL_2(11)$, see \cite[p. 18]{conway1985atlas}, and hence $X$ is conjugate to $\Aut(\E)_{\b 0}$. This implies that $X$ fixes a codeword, which is a contradiction.
 
 Suppose $\E$ is $(X,3)$-neighbour transitive, for some $X\leq \Aut(\E)$. Then, by Lemma~\ref{design}, the weight six codewords of $\E$ would form a $3$-$(11,6,\lambda)$ design, for some integer $\lambda$. The equation $$b=\frac{v(v-1)(v-2)}{k(k-1)(k-2)}\lambda$$ gives $\lambda=11.6.5.4/(11.10.9)=4/3$, a contradiction. Since $\delta\geq 6$, it follows that $\rho\geq 3$ and hence $\E$ is not completely transitive. 
\end{proof}

\begin{proposition}\label{equivpunchad}
 Let $C$ be a $2$-regular code in $H(11,2)$ with $\delta\geq 5$ and $|C|\geq 2$. Then one of the following holds: 
 \begin{enumerate}
 \item $\delta=11$ and $C$ is equivalent to the binary repetition code,
 \item $\delta=5$ and $C$ is equivalent to $\P$, or 
 \item $\delta=6$ and $C$ is equivalent to $\mathcal{E}$.
 \end{enumerate}
\end{proposition}

\begin{proof}
 Suppose $C$ is a $2$-regular code. By Remark~\ref{remfav} we can assume $\b 0\in C$.  Weight $\delta$ codewords must exist, and form a $2$-$(11,\delta,\lambda)$ design, by Lemma~\ref{design}. If $\delta=11$ we get the trivial design with one block, that is $|C|=2$ and $C$ is equivalent to the binary repetition code. Suppose $\delta<11$. Using the equation $vr=bk$, we have $11r=b\delta$. So, since $\delta< 11$, $11$ divides $b$. Now, Table~1 in \cite{Best78boundsfor} gives $|C|\leq 24$, so that $b=11$ or $22$. The equation $$b=\frac{v(v-1)}{k(k-1)}\lambda=\frac{11.10}{\delta(\delta-1)}\lambda$$ then gives $a\delta(\delta-1)=10\lambda$, where $a=1$ or $2$, so that $\delta=5,6$ or $10$. If $\delta=10$ then $a=1$, $\lambda=9$ and the design consists of every weight $10$ vector. However, $d((1^{10},0),(0,1^{10}))=2$, which contradicts $\delta=10$. 
 
 Suppose $\delta=5$. Then the weight $5$ codewords of $C$ form a $2$-$(11,5,\lambda)$ design, and since $a.5.4=10\lambda$, we see $\lambda=2$ or $4$. Consider distinct weight $5$ codewords $\alpha,\beta$ which share $2$ non-zero entries. Without loss of generality, $\alpha=(1^5,0^6)$ and $\beta=(1^i,0^{5-i},1^{5-i},0^{i+1})$, for some $i\in\{2,3,4\}$. Now, $d(\alpha,\beta)=10-2i\geq\delta=5$, so $i=2$ and $d(\alpha,\beta)=6$. Since the code is $2$-regular, and in particular $0$-regular, this implies there are weight $6$ vertices in the code. Hence the weight $6$ codewords form a $2$-$(11,6,\lambda')$ design, with the equation $vr=bk$ implying $11$ divides $b$. The bound $|C|\leq 24$ then implies that $b=11$ in both designs. Hence $\lambda'=3$ and $\lambda=2$ and the designs are equivalent to $\D$ and its complementary design respectively (see Definition~\ref{hadamarddef}). As we have just seen, it is impossible to form a design from weight $k$ vertices for $k\neq 0,5,6,11$. Hence the distance distribution of $C$ is $(1,0,0,0,0,11,11,0,0,0,0,b)$ where $b=0$ or $1$. Taking the MacWilliams transform when $b=0$ gives $a_2'<0$ which contradicts \cite[Lemma~5.3.3]{van1999introduction}, and thus $b=1$. Regularity then implies the designs are indeed the complementary designs of each other, not just equivalent to these, hence $C=\P$. 
 
 Suppose $\delta=6$. By a similar argument to above we see that the weight $6$ codewords form a $2$-$(11,6,\lambda)$ design. Table~1 in \cite{Best78boundsfor} gives $|C|\leq 12$, and so we see that $b=11$ and $\lambda=3$. Thus $C$ consists of the zero codeword and the characteristic vectors of the blocks of $\D$. 
\end{proof}

%
%

Given the results of the previous section we are left with the possibility that $X$ and $X_{\b 0}$ have different socles. The next proposition gives a classification of the inclusions of $2$-homogeneous groups in $2$-transitive groups with different socles.

\begin{proposition}\label{2tr2hominc}
 Let $G,H$ be groups such that $G < H\leq S_m$, $H$ is $2$-transitive of degree $m$ with $\soc(H)\neq A_m$, and $G$ is $2$-homogeneous of degree $m$ with $\soc(G)\neq \soc(H)$. Then $G$, $H$, and $m$ are as in Table~\ref{2tr2hom}.
\end{proposition}

\begin{table}
\begin{center}
\begin{tabular}{cccc}
\hline\noalign{\smallskip}
 $G$ & $H$ & $m$ & Table of \cite{Liebeck1987365}\\
 \noalign{\smallskip}\hline\noalign{\smallskip}
 $Z_7.Z_3$ & $\PSL_3(2)$ & $7$ & I\\
 $Z_{11}.Z_5$ & $\PSL_2(11)$ or $M_{11}$ & $11$ & I\\
 $Z_{23}.Z_{11}$ & $M_{23}$ & $23$ & I\\
 $\PSL_2(7)$ & $\AGL_3(2)$ & $8$ & II\\
 $A_7$ & $A_8$ & $15$ & III\\
 $\PSL_2(11)$ & $M_{11}$ & $11$ & IV\\
 $\PSL_2(11)$ or $M_{11}$ & $M_{12}$ & $12$ & IV\\
 $\PSL_2(23)$ & $M_{24}$ & $24$ & IV\\
 \noalign{\smallskip}\hline
\end{tabular}
\caption{Groups $G<H\leq S_m$ where $H$ is $2$-transitive, $G$ is $2$-homogeneous, $\soc(H)\neq A_m$ and $\soc(G)\neq \soc(H)$.}\label{2tr2hom}
\end{center}
\end{table}

\begin{proof}
 We use the classification of $2$-transitive groups, \cite[Section 7.7]{dixon1996permutation}, and \cite[Theorem 9.4B]{dixon1996permutation} which gives those groups which are $2$-homogeneous, but not $2$-transitive. Now, $G$ and $H$ are $2$-homogeneous and $2$-transitive, respectively, and thus either almost-simple or affine. Moreover, $G$ is not maximal in $A_m$ or $S_m$, since then $\soc(H)=A_m$. If $G$ is affine, part (I) of the main result from \cite{Liebeck1987365} tells us that $G$ is in \cite[Table I]{Liebeck1987365} of type (c), giving the examples in Table~\ref{2tr2hom}, lines 1--3 (we note that the group $Z_{17}.Z_8$, in $H=\PGaL_2(16)$ is not $2$-homogeneous). Suppose $G$ is almost simple. If $H$ is affine then part (II) of the main result from \cite{Liebeck1987365} tells us that $H$ is of type (c) and listed in \cite[Table~2]{Liebeck1987365}, giving Table~\ref{2tr2hom}, line 4. If $H$ is almost simple, part (II) of the main result from \cite{Liebeck1987365} gives us that $\soc(G)$ and $\soc(H)$ are listed in \cite[Tables III, IV, V, and VI]{Liebeck1987365}, giving  Table~\ref{2tr2hom}, lines 5--8. 
\end{proof}

\begin{table}
\begin{center}
\begin{tabular}{ccccccc}
\hline\noalign{\smallskip}
$X_{\b 0}$ & $X$ & $m$ & $q$ & $|X_{\b 0}|$ & $\binom{m}{2} (q-1)^2$ & Lemma~\ref{codeorder} \\
 \noalign{\smallskip}\hline\noalign{\smallskip}
 $Z_7.Z_3$ & $\PSL_3(2)$ & $7$ & $2$ & $3.7$ & $3.7$ & \\
 $Z_{11}.Z_5$ & $\PSL_2(11)$ & $11$ & $2$ & $5.11$ & $5.11$ & \\
  & $M_{11}$ & $11$ & $2$ & & $5.11$ & \\
 $A_7$ & $A_8$ & $15$ & $7$ & $2^3.3^2.5.7$ & $2^2.\b{3^3} .5.7$ & fail\\
  & & & $8$ & & $3.5.\b{7^3}$ & fail\\
 $\PSL_2(11)$ & $M_{11}$ & $11$ & $2$ & $2^2.3.5.11$ & $5.11$ & \\
  & & & $10$ & & $\b{3^4} .5.11$ & fail \\
  & $M_{12}$ & $12$ & $11$ & & $\b{2^3}.3.\b{5^2} .11$ & fail \\
  & & & $12$ & & $2.3.\b{{11}^3}$ & fail \\
 $M_{11}$ & $M_{12}$ & $12$ & $11$ & $2^4.3^2.5.11$ & $2^3.3.\b{5^2} .11$ & fail \\
  & & & $12$ & & $2.3.\b{{11}^3}$ & fail \\
 $\PSL_2(23)$ & $M_{24}$ & $24$ & $23$ & $2^4.3.11.23$ & $2^4.3.\b{11^2} .23$ & fail \\
 \noalign{\smallskip}\hline
\end{tabular}
\caption{Entry faithful groups $X,X_{\b 0}\leq \Aut(\varGamma)$ such that $X$ is almost simple $2$-transitive on $M$, $X_{\b 0}$ is $2$-homogeneous on $M$, $\soc(X)\neq A_m$ and $\soc(X_{\b 0})\neq \soc(X)$.}
\label{efgroups}
\end{center}
\end{table}

We are now in a position to prove the main results.

\begin{proof}[Proof of Theorem~\ref{main}]
 By Remark~\ref{remfav} we may assume $\b 0\in C$. By Corollary~\ref{2transM} and Proposition~\ref{ihom}, we have that $X$ acts $2$-transitively on $M$ and $X_{\b 0}$ acts $2$-homogeneously on $M$. By Lemma~\ref{equalsocles} if $\soc(X)=\soc(X_{\b 0})$ then $C$ is the binary repetition code and $X$ is one of the groups listed in Table~\ref{repetitiongroups}. Suppose $\soc(X)\neq\soc(X_{\b 0})$. Since $X^M$ is $2$-transitive, $X$ is either affine or almost simple. Suppose $X$ is affine. Then, by Proposition~\ref{2tr2hominc}, $m=8$ and $X\cong \AGL_3(2)$. Hence $X_1\cong \PSL_3(2)\cong\PSL_2(7)$, acting $2$-transitively on $Q$, and so $q=7$ or $8$. Since $\delta\geq 5$, $X_{\b 0}$ must be transitive on $\varGamma_2(\b 0)$. However, $|\varGamma_2(\b 0)|=m(m-1)(q-1)^2/2=4.7.6^2$ or $4.7^3$ respectively, neither of which divides $|X|=8.7.3$. Thus $X$ is almost simple. 
 
 If $\soc(X)=A_m$ then, by Lemma~\ref{amsm}, $C$ is equivalent to the binary repetition code, and so $X=S_m$, $X_{\b 0}=A_m$, contradicting the assumption that $\soc(X)\neq \soc(X_{\b 0})$. Thus $\soc(X)\neq A_m$. Since $X$ is almost-simple and $2$-transitive on $M$, $X_{\b 0}$ is $2$ homogeneous on $M$, with $\soc(X)\neq A_m$ and $\soc(X)\neq \soc(X_{\b 0})$, we can apply Lemma~\ref{2tr2hominc}. Hence, $X,X_{\b 0}$ and $m$ are as in Table~\ref{2tr2hom}, with $X=H$ and $X_{\b 0}=G$. Moreover, by Remark~\ref{imprimrank3}, the values of $X$, $m$, and $q$ appear (as $G$, $n$, and $|B|$ respectively) in one of the lines of \cite[Tables~2 or 3]{devillers2011imprimitive}. The groups which appear both in Table~\ref{2tr2hom}, and also in one of the lines of \cite[Tables~2 or 3]{devillers2011imprimitive}, are listed in Table~\ref{efgroups}, along with the corresponding value of $q$, noting that Lemma~\ref{codeorder} tells us that if $|X_{\b 0}|=m(m-1)/2$ then $q=2$. Furthermore, Lemma~\ref{codeorder} allows us to obtain the last column in Table~\ref{efgroups}, where a group is marked with `fail' if it does not satisfy the first part of this result. Dealing with the remaining cases here will complete the proof. In each case $q=2$.
 
 Suppose $m=7$, $X_{\b 0}=Z_7.Z_3$ and $X=\PSL_3(2)$. It follows that $|C|=8$. Since $\delta\geq 5$, there are $7$ non-zero codewords, each having weight either $5,6,$ or $7$. Since $m=7$, the supports of two such codewords intersect in at least three entries. Hence the distance between two such codewords is at most $4$, a contradiction. Suppose $m=11$, $X_{\b 0}=Z_{11}.Z_5$ and $X=\PSL_2(11)$. Then $|C|=12$, so Proposition~\ref{equivpunchad} says that $C$ is equivalent to the even weight subcode of the punctured Hadamard $12$ code, however this contradicts Lemma~\ref{punchadamard}. Suppose $m=11$, $X_{\b 0}=Z_{11}.Z_5$ and $X=M_{11}$. Then $|C|=2^4.3^2$, which contradicts the Singleton bound $|C|\leq 2^{m-\delta+1}\leq 2^7$. If $m=11$, $X_{\b 0}=\PSL_2(11)$ and $X=M_{11}$, then $|C|=12$. By Proposition~\ref{equivpunchad}, $C$ is equivalent to the even weight subcode $\E$ of the punctured Hadamard $12$ code. Moreover, by Lemma~\ref{punchadamard}, $C$ is indeed $(X,2)$-neighbour transitive. 
\end{proof}

\begin{table}
\begin{center}
\begin{tabular}{ccccccc}
\hline\noalign{\smallskip}
$\delta=m$ & $5$ & $6$ & $7$ & $8$ & $9$ & $10$ \\
 \noalign{\smallskip}\hline\noalign{\smallskip}
 $X$ & -- & $\PGL_2(5)$ & $\AGL_1(7)$ & $\PGL_2(7)$ & $\leq\AGL_2(3)$ & $\leq\PGaL_2(9)$\\
 \noalign{\smallskip}\hline
\end{tabular}
\caption{Groups $X\neq S_m$ from Table~\ref{repetitiongroups}, for $m\leq 10$.}
\label{homgroups}
\end{center}
\end{table}

\begin{proof}[Proof of Theorem~\ref{comptrans}]
 Suppose $q=2$ and $C$ is equivalent to the binary repetition code. We show that $C$ is $X$-completely transitive with $X\cong S_m$.  It is clear that $L(\cong S_m) \leq\Aut(C)$ and that $H=\langle(h,\ldots,h)\rangle\leq\Aut(C)$, where $1\neq h\in S_2$. Now let $X$ be the group consisting of automorphisms of the form $x=(h,\ldots,h)\sigma$ if $\sigma$ is an odd permutation and $x=\sigma$ if $\sigma$ is an even permutation.  Then $X\cong S_m$, $X_{\b 0}\cong A_m$, $X\cap S_q^m=1$, and $X$ acts transitively on $C$.  The covering radius of $C$ is $\lfloor\frac{m}{2}\rfloor$ and $C_i$ consists of the vertices of weights $i$ and $m-i$, for $i=0,\ldots,\lfloor\frac{m}{2}\rfloor$.  Let $\nu_1,\nu_2\in C_i$.  If $\nu_1,\nu_2$ both have the same weight, then because $A_m$ acts $i$-homogeneously on $M$ for all $i\leq m$ it follows that there exists $\sigma\in X_{\b 0}$ such that $\nu_1^\sigma=\nu_2$.  Now suppose $\nu_1$ and $\nu_2$ have different weights, say $\nu_1$ has weight $i$ and $\nu_2$ has weight $m-i$.  Then there exists $x\in X$ such that $\nu_2^x$ has weight $i$.  Consequently there exists $\sigma\in X_\alpha$ such that $\nu_1^\sigma=\nu_2^x$, thus $\nu_1^{\sigma x^{-1}}=\nu_2$.  Hence $X$ acts transitively on $C_i$ and so $C$ is $X$-completely transitive.  
 
 Conversely, suppose $C$ is entry-faithful $X$-completely transitive, for some group $X$. Since $\delta\geq 5$, and so $\rho\geq 2$, $C$ is $(X,2)$-neighbour transitive. Thus, by Theorem~\ref{main}, either $C$ is the even weight subcode $\E$ of the Hadamard code of length 12, or $C$ is the binary repetition code. By Lemma~\ref{punchadamard}, $\E$ is not $X$-completely transitive for any group $X$. Thus, $q=2$ and $C$ is equivalent to the binary repetition code, and $X$ is listed in Table~\ref{repetitiongroups}. Note that in this case $\soc(X)=\soc(X_{\b 0})$. If $\soc(X)=A_m$ then $X=S_m$, $X_{\b 0}=A_m$ and $C$ is $X$-completely transitive by the above. 
 
 Suppose that $\soc(X)\neq A_m$. By Proposition \ref{ihom}, $X_{\b 0}$ (and so $X$ also) is $i$-homogeneous on $M$ for all $i\leq \lfloor\frac{\delta-1}{2}\rfloor=\lfloor\frac{m-1}{2}\rfloor$. By \cite[Section 7.7 and Theorem 9.4B]{dixon1996permutation}, we see that the groups in Table~\ref{repetitiongroups} are at most $4$-homogeneous, and so $m\leq 10$. Comparing the possible values for $m$ with those groups in Table~\ref{repetitiongroups} satisfying the conditions of the last column, we arrive at Table~\ref{homgroups}. By \cite[Section 7.7 and Theorem 9.4B]{dixon1996permutation}, we see that $\AGL_1(7)$, $\AGL_2(3)$ and $\PGaL_2(9)$ do not have $\lfloor\frac{\delta-1}{2}\rfloor$-homogeneous subgroups. Suppose $X=\PGL_2(7)$ and $m=8$. Then $X$ must be transitive on the set of weight $4$ vertices, of which there are $\binom{8}{4}=70$ such vertices. However, $5$ does not divide $|X|$. 
 
 Suppose $X\cong \PGL_2(5)$ and $m=6$. Then $X_{\b 0}\cong \PSL_2(5)$. Since $q=2$ we have $X_{\b 0}=X\cap L\leq L\cong S_m$. Let $h=((01),\ldots,(01))\in S_2^6$. As every $x\in X\setminus X_{\b 0}$ must map the zero codeword to the all ones codeword, it follows that for some $\sigma\in\PGL_2(5)\setminus \PSL_2(5)\subseteq X^M$ we have $x=h\sigma$. By \cite[Section 7.7 and Theorem 9.4B]{dixon1996permutation}, $\PGL_2(5)$ is $3$-transitive, but $\PSL_2(5)$ is not $3$-homogeneous. In fact, by \cite[Sec. 2.4 and Lem. 9.1.1]{neilphd}, $\PSL_2(5)$ has two orbits on the set of weight $3$ vertices, and these are the characteristic vectors of complementary $2$-$(6,3,2)$ designs. Denote these orbits as $\O_1,\O_2$, and note that for $\sigma\in\PGL_2(5)\setminus\PSL_2(5)\subseteq \Sym(M)$ we have $\O_1^\sigma=\O_2$. Since the designs are complementary, we have $\O_1^h=\O_2$. Thus, for $x=h\sigma\in X\setminus X_{\b 0}$, where $\sigma\in\PGL_2(5)\setminus\PSL_2(5)$, we have $\O_1^x=\O_1^{h\sigma}=\O_1$. Hence $X$ is not transitive on the set of weight $3$ vertices. 
\end{proof}

\end{document}